\documentclass[11pt]{article}
\usepackage[utf8]{inputenc}
\usepackage[english]{babel}
\usepackage{amsfonts}
\usepackage{amssymb}
\usepackage{amsmath}
\usepackage{amsthm}

\newtheorem{definition}{Definition}
\newtheorem{proposition}{Proposition}
\newtheorem{theorem}{Theorem}

\newtheorem{lemma}{Lemma}
\newtheorem{corollary}{\noindent Corollary}
\newtheorem{problem}{\noindent Problem}

\title{Three theorems on the uniqueness of the Plancherel measure from different viewpoints\thanks{The work was carried out at the Institute for Information Transmission Problems, and was funded by the Russian Science Foundation (Project no. 14-50-00150).}}

\author{A.~M.~Vershik\thanks{St.~Petersburg Department of Steklov Institute of Mathematics,
St.~Petersburg State University, and Institute for Information Transmission Problems.
E-mail: {\tt avershik@gmail.com}.}}

\date{October 2, 2018}

\begin{document}

\maketitle

\begin{flushright}{\emph{To V.~M.~Buchstaber with friendly wishes on the occasion of his 75th birthday}}
\end{flushright}

\begin{abstract}
We consider three uniqueness theorems: one from the theory of meromorphic functions, another one from asymptotic combinatorics, and the third one about representations of the infinite symmetric group.

The first theorem establishes the uniqueness of the function~$\exp z$ in a class of entire functions.
The second one is about the uniqueness of a random monotone nondegenerate numbering of the two-dimensional lattice~$\Bbb Z^2_+$, or  of a nondegenerate central measure on the space of infinite Young tableaux.
And the third theorem establishes the uniqueness of a representation of the infinite symmetric group~${\frak S}_{\infty}$ whose restrictions to finite subgroups have vanishingly few invariant vectors.

But in fact all three theorems are, up to a nontrivial rephrasing of conditions from one area of mathematics in terms of another area, the same theorem! Up to now, researchers working in each of these areas were not aware of this equivalence. The parallelism of these uniqueness theorems on the one hand, and the difference of their proofs on the other hand call for a deeper analysis of the nature of uniqueness and suggest to transfer the methods of proof between the areas.

More exactly, each of these theorems establishes the uniqueness of the so-called Plancherel measure, which is the main object of our paper. In particular, we show that this notion is general for all locally finite groups.
\end{abstract}

\section{Introduction}
The term ``Plancherel measure'' for the infinite symmetric group was suggested by me and introduced in our joint paper with S.~V.~Kerov~\cite{VK-77}. Later, this notion proved to be extremely important not only in the original context. {\it In the asymptotic representation theory} of inductive families of groups and algebras it plays the same role as Haar measure in the theory of locally compact groups. In the present paper, we also illustrate this by examples, in particular, by showing how diverse are situations where it appears. But the motivation behind the choice of the name for this measure does not fully agree with the conventional use of this term; this question is discussed below. This difference is especially clear in the combinatorial definition of the Plancherel measure as a measure on lattice numberings.

With regard to the theory of symmetric groups, it is natural to compare the Plancherel measure with the Poisson--Dirichlet measure, which arises, as shown in our previous papers~\cite{VSh1, VSh2}, in the study of limit distributions of conjugacy classes (i.e., cycle lengths of random permutations). We describe completely different situations in which this measure arises implicitly or explicitly. To understand the overall picture of asymptotic theory, it is of importance to consider the Plancherel measure and its properties as an example of a distinguished central measure on the path space of a graded graph, or as an example of a special Gibbs measure of maximal entropy in constructions of statistical physics.

We begin with a theorem having the most beautiful and clear statement and the most difficult known direct proof. Namely, Edrei's theorem on the uniqueness of the exponential function as a totally positive entire nonconstant function without zeros, see~\cite{Edr,W,ASM}, early papers by Schoenberg~\cite{Sch}, papers by Krein and Gantmacher~\cite{GK}, the book~\cite{Kar}, etc.

Then we turn to the main part of the paper, where we describe the Plancherel measure from different viewpoints and consider the uniqueness theorem for this measure as a nondegenerate central measure on the path space of the Young graph and as a nondegenerate homogeneous monotone random numbering of the two-dimensional lattice~${\Bbb Z}_+^2$.

The interest to the Plancherel measure is due to the fact that it  became the source of uniting such different facts from the theory of entire functions, representation theory, combinatorics, and the theory of invariant measures. That is why we discuss in detail the group nature of Plancherel measures for locally compact groups, both in the most general context and in the case of locally finite groups and, in particular, the infinite symmetric group. Here, the main uniqueness theorem looks as a theorem on nondegenerate ergodic central measures on the path space of the Young graph.

We treat separately the combinatorial definition of the Plancherel measure as the limit of Plancherel measures of finite groups; in particular, in the case of symmetric groups this is a measure on lattice numberings.

It is important to note that, in contrast to the finite case and the usual understanding of Plancherel measures, this measure can be defined meaningfully not on the space of classes of irreducible representations (Young diagrams in the case of symmetric groups), but only on the space of tableaux, i.e., sequences of representations.

The uniqueness problem is the deepest and most difficult question; it is convenient to state it in a general form, as a purely combinatorial problem, for instance, as a theorem on numberings of multidimensional lattices. It would not be an overstatement to compare the existence and uniqueness theorem for numberings of the two-dimensional lattice with  the existence and uniqueness theorem for Haar measure.

The third uniqueness theorem is a rephrasing of the previous one in terms of representations of the infinite symmetric group and their restrictions to finite subgroups. Namely, it states that the regular representation of the group~$\frak{S}_{\Bbb N}$ is distinguished among all representations with character. Here, the focus moves to the asymptotics of the first row of a random Young diagram, whose role in the asymptotic representation theory of symmetric groups  is well known.

In this paper, we do not discuss numerous proofs of these theorems (their list is given below), leaving such a discussion to another occasion, partly because the author does not think that we know the ``proof from the Book'' (in the sense of P.~Erd\"os) for any of these theorems.

The paper is an extended version of my talk at the conference {\it Algebraic Topology, Combinatorics, and Mathematical Physics} (Moscow, May 2018) dedicated to the 75th anniversary of V.~M.~Buchstaber.

\section{A nonvanishing totally positive entire function is the exponential function}

We begin with  the most clear uniqueness theorem, essentially belonging to classical complex analysis.

\begin{theorem}[A.~Edrei, M.~Assein, A.~Whitney, I~Schoenberg \cite{ASM,Edr,W}]
Let $\varphi(z)$ be an entire function on the complex line~$\Bbb C$,
 $$
\varphi(z)=\sum_{n=0}^{\infty}c_n z^n, \qquad c_{-n}=0 \quad \text{for all}\quad n>0,
$$
satisfying the following conditions.

{\rm 1)} \emph{Total positivity:} for every $n\in {\Bbb Z}$, for every collection of indices $\{i_1,\dots, i_n; j_1, \dots, j_n\}$ where $i_s,j_s \in {\Bbb Z}$,
$$
\det ||c_{i_s-j_t}||_{s,t=1,\dots, n} \geq 0.
$$

{\rm 2)} \emph{Nondegeneracy:} the function does not vanish, $\varphi(z)\ne 0$;  and (to simplify the statement) \emph{normalization:}
$\varphi(0)=\varphi'(0)=1$.

Then
 $$\varphi(z)=\exp z.
 $$

If, keeping the total positivity condition, we assume that  $\varphi(z)$ is only meromorphic (i.e., allow it to have poles) and waive the nonvanishing condition and the normalization, then the general form of such a function is as follows:
  $$
  \varphi(z)=z^m e^{\gamma z}\prod_{n=1}^{\infty}\frac{1+\alpha_n z}{1-\beta_n z},
 $$
where $\alpha_1\geq \alpha_2 \geq \dots \geq 0$; $\beta_1\geq \beta_2 \geq \dots \geq 0$; $\gamma \geq 0$;
   $\sum_n \alpha_n +\sum_n \beta_n +\gamma =1$; $m\in {\Bbb N}_+$.
 \end{theorem}

There is a similar theorem for Laurent series, but we will not consider it. The notion of total positivity appeared in the 1930s in works by Schoenberg and von Neumann (on the geometry of Hilbert spaces), and simultaneously in pre-war papers by M.~G.~Krein and F.~Gantmacher~\cite{GK} in connection with functional analytic methods in the study of differential equations. It also arises in statistics and the theory of convex sets (the so-called Chebyshev curves), see S.~Karlin's book~\cite{Kar}. In recent years, the interest to total positivity has considerably increased in connection with algebraic and representation-theoretic problems (papers by G.~Lusztig~\cite{Lus}, A.~Zelevinsky and S.~Fomin~\cite{FZ}, and many others). The recent paper~\cite{BGl} by V.~M.~Buchstaber and A.~Glutsyuk also uses this notion. But here the notion is greatly extended: one considers not Toeplitz matrices as above, but arbitrary matrices. From this point of view, it would be interesting to consider this notion for arbitrary groups (rather than only~$\Bbb Z$ and $\Bbb R$), see below.

In E.~Thoma's paper~\cite{T}, total positivity appeared as follows (to the author's knowledge, this connection was known already to M.~G.~Krein). Consider a positive definite (in the ordinary sense) normalized central function~$\chi(\cdot)$ on the infinite symmetric group~${\frak S}_{\Bbb N}$, i.e., a character of this group. Assume that it is indecomposable, i.e., cannot be written as a nontrivial convex combination of other characters, and restrict this character (regarded as a function on the group) to the conjugacy classes of cycles, i.e., to the semigroup~$\Bbb Z_+$.\footnote{A character $\chi(\cdot)$ takes the same value~$\chi(n)$ at all cycles of length~$n$, hence it produces a function on~$\Bbb Z$ supported by~${\Bbb Z}_+$.} Then the function
    $$ \exp{\sum_{n\geq 1} \frac{\chi(n)}{n}z^n}$$
of a variable~$z$,  i.e., the exponential of the generating function of these values on cycles,  is a totally positive function in the above sense. Here it is important to keep in mind the multiplicativity property of indecomposable characters (see~\cite{VK}): the values of such a character on an arbitrary conjugacy class is the product of its values on the corresponding cycles, i.e., the values on cycles uniquely determine the character. Thus, the positive definiteness of an indecomposable character (regarded as a function on the group~${\frak S}_{\Bbb N}$) can be restated in terms of its restriction to cycles as the total positivity. It is this observation and the multiplicativity property of indecomposable characters (which we do not need now) that allowed E.~Thoma to use the previous theorem to give for the first time a formula for the characters of~${\frak S}_{\Bbb N}$. Thoma did not know about Edrei's theorem, but used a close theorem by Bieberbach, which, together with his own considerations, gave an equivalent result. Thoma's paper contains no ``representation-theoretic'' interpretation of the parameters and no realization of representations.

Later, a ``representation-theoretic'' proof of the theorem was given in~\cite{VK}, as well as an interpretation of the parameters and a realization of the factor representations. By now, there are already five different proofs (see below), but  still no purely combinatorial proof, which should clarify important properties of general lattices.

In this connection, we pose the following problem, which, apparently, has not been studied. As mentioned above, instead of an entire or meromorphic function, or even a Laurent series, one can consider the sequence of its Taylor (Laurent) coefficients as a function on~$\Bbb Z$, and call this sequence a positive definite function on the group~$\Bbb Z$. This justifies a general definition of a totally positive function on a locally compact, or even arbitrary, group or even semigroup with given linear order. For group $\Bbb Z$, or $\Bbb R$ this is usual linear order; for general countable group this is a numeration of he group  using integers $\Bbb Z$ or $\Bbb N$. For example the semigroup ${\Bbb Z}^2_+$ a numeration is defined by infinite Young tableau which cover
all semigroup.

Namely, in the definition of {\it total positive definiteness} of function on the group or semigroup, we require the nonnegativity of all, and not only principal, minors with given order of variables (elements of the group or semigroup). This means that for any two finite subsets of the elements of the group with the same number of elements,
$\{g_i\}_{i=1,\dots n}, \{h_j\}_{j=1,\dots n}$, with have monotonic numeration with respect to the given linear order we require non-negatively ofthe minor $\|f(g_i\cdot h^{-1}_j)\|_{i,j}$.
Evidently the answer depends on the orders.

There appears a series of questions which are not clear even for $\Bbb Z$ with standard order. For instance, what unitary representations and what vectors in these representations correspond (in the sense of the GNS construction) to totally positive functions? These representations are most likely reducible, and these vectors have some additional properties. For the group~$\Bbb Z$, such a representation is regular, but is  is not known what property distinguishes the vector states corresponding to totally positive matrix elements; this is an interesting question close to the problem of central measures.

\section{Plancherel measure}

The definition of the Plancherel measure for the infinite symmetric group as a measure on the space of infinite Young tableaux, which was given in our paper~\cite{VK-77} (see~\cite{VK-81,VK-85} and also~\cite{Shepp}), became established in the literature, but it is  long  in need of a more complete justification. Moreover, in the context under consideration, the term ``Plancherel measure'' itself can be disputed, since, in a sense, it does not coincide with the canonical understanding of this term. That is why we make several general comments on the considerations we had in mind, on numerous links to concrete studies of traditional Plancherel measures (for example, to the 1960s papers by  Gelfand's school on Plancherel measures for some semisimple groups, which became classical), on the relation to general representation theory. Then we return to our main case of interest, that of locally finite groups and, in particular, the infinite symmetric group. In short, we use this term, first, because the Plancherel measure in our sense is the limit of Plancherel measures for finite subgroups; and, second, because for the limiting group  it can be regarded as the most natural candidate for such a measure in the classical sense; recall that the group is not of  type~I, so the notion itself is not uniquely defined. The corresponding uniqueness theorem should be understood with regard to these considerations.

\subsection{Plancherel measure in the general context}

By a Plancherel measure one usually means a measure on the space of equivalence classes of irreducible representations of a group.

For a {\it commutative locally compact group}, the Plancherel measure is the Haar measure on the group of characters. For a {\it finite group}, the probability measure on the space of equivalence classes of irreducible complex representations whose value at each representation~$\pi$ is proportional to the squared dimension of~$\pi$  is distinguished and should be called the Plancherel measure by analogy with the classical Plancherel measure.

For a {\it nondiscrete compact group}, the Plancherel measure is $\sigma$-finite, and its weights, more exactly, the ratios of the measures of two irreducible reprentations, are still equal to the ratio of their squared dimensions. The duality in this case is well known starting from works by Tanaka and Krein.

In all these cases, there is an analog of Parceval's identity (or the Plan\-che\-rel isometry) for the group~$\Bbb R$, i.e., an isometry between the group algebra
 ${{\Bbb C}[G]\sim l^2[G]}$ with the Hilbert structure determined by the inner product corresponding to the Haar measure (the uniform measure on~$G$) and the space~$C(G)$  of compex-valued functions on~$G$ with the Hilbert structure generated by the Plancherel measure on the space of equivalence classes of representations.

This isometry can be expressed as the following equality of inner products:
$$\langle f_1,f_2\rangle_{\rm Ha} = \langle R(f_1),R(f_2)\rangle_{\rm Pl};$$
here $f_1,f_2 \in {\Bbb C}[G]$ and $R$ is the right regular representation of the algebra on itself,  $\langle\cdot,\cdot\rangle_{\rm Ha}$ is the inner product corresponding to the Haar measure on the group~$G$, and  $\langle\cdot,\cdot\rangle_{\rm Pl}$ is the inner product defined as the normalized direct sum of the inner products in each of the minimal primary subrepresentations, whose spaces are the matrix algebras~$M_k({\Bbb C})$ with the ordinary inner product $\langle A,D\rangle=\operatorname{Tr}\{AB^*\}$.

For type~I groups, the definition of the Plancherel measure reduces to the same formula, which uses the pairing of operator functions corresponding to the representation; such a definition is essentially given in Dixmier's book~\cite{Dix} for type~I groups (in his terminology, GCR groups), which exhausts the problem of defining the Plancherel measure for these groups.

But there is another important circumstance, which was discussed in the old literature. By the decomposition of the regular representation one should mean the {\it central decomposition of the regular representation}. It is uniquely defined: this is the decomposition into factor representations, and for type~I groups they are multiples of irreducible representations  (primary). Hence, it is natural to regard the Plancherel measure as a measure on primary representations, which does not change matters much, since there arise only multiplicities of irreducible representations. Therefore, the uniqueness of the Plancherel measure in this sense is ensured by an agreement about the choice of additional normalizations depending on the relation between multiplicities.

Let us add that for nonamenable groups, representations that are not weakly contained in the regular representation (e.g., complementary series representations for semisimple groups) do not belong to the support of the Plancherel measure and are not related to it in any sense.

Only the central decomposition of a representation, i.e., the decomposition with respect to the center of the von Neumann algebra generated by the image of the group algebra under the representation, is canonical; this is a decomposition into factor representations (and not in irreducible representations). If the von~Neumann algebra generated by a representation is of type~II or~III, then its factor representations are no longer primary, i.e., multiples of irreducibles. The central decomposition is a necessary step in harmonic analysis, but, of course, it does not give complete information on the representation.

The difficulty with the definition of the Plancherel measure for groups not of type~I is due to the fact that if we want to obtain a measure on the decomposition of a representation into irreducible representations, i.e., on the dual object of a locally compact group, then we must have a Borel structure on this object. But such a structure is not available, so we are forced to choose a part of the set of equivalence classes of irreducible representations. This choice, according to the general theory, is performed by choosing a maximal commutative self-adjoint subalgebra in the commutant of the algebra generated by the representation.

But if we take such a maximal commutative subalgebra in the commutant of the algebra generated by a factor representation, then its  decomposition into irreducible representations depends substantially on the choice of this subalgebra. It is known that these decompositions can even be disjoint for two different maximal subalgebras. This question is still poorly studied.

   \subsection{Plancherel measure for locally finite groups}

Recall that the infinite symmetric group is not a group of type~I, but, on the contrary,  is rather a typical group of type~II or~III (or
 NCGR according to Dixmier~\cite{Dix}).

In the framework of asymptotic representation theory, a natural definition of the Plancherel measure is based on considerations quite different from those described above, namely, approximative considerations. Indeed, the analysis of finite groups leads to the conclusion that in the case of an inductive limit of finite groups (yielding a group not of type~I), the natural candidate for the Plancherel group is the inductive limit of the Plancherel measures on the prelimit groups. Moreover, this measure should be defined not on the space of equivalence classes of all irreducible representations, which does not exist in a reasonable sense, but on another object, which will be discussed below.

Let a group $G_{\infty}$ be defined as an increasing limit of finite groups:
   $$G_{\infty}=\mbox{lim-ind}_n G_n.$$
The dual spaces ${\hat G}_n$ (the sets of equivalence classes of irreducible representations of~$G_n$) form an inverse spectrum, whose limit could be regarded as the dual object to the inductive limit of groups:
 $$\hat G= \mbox{lim-proj}_n {\hat G}_n.$$
However, according to the general theory, points of the projective limit are coherent (with respect to the projections) sequences of points of prelimit spaces, i.e., in terms of our constructions, sequences of representations~$\pi_n$ of~$G_n$ such that $\pi_n$ is an irreducible representation of the group~$G_n$ occurring as a component in the representation induced by the representation~$\pi_{n-1}$ of the group~$G_{n-1}$. In terms of symmetric groups, these are tableaux. Thus, we arrive at the conclusion that the Plancherel measure of the limiting group is, possibly, a measure defined on the space of sequences of partitions. The collection of Plancherel measures on the spaces of representations of
$G_n$  can be regarded as a coherent sequence of points of the corresponding simplices, and the limit point  is a point of  the projective limit of simplices, which, in turn, is also a Choquet simplex (see~\cite{VPr}). We will show that the sequence of Plancherel measures does indeed determine a point of the projective limit, which will be called the Plancherel measure of the limiting group~$G_{\infty}$.

 \begin{lemma}
Let $G$ be a finite group and $H\subset G$ be a subgroup of~$G$, let $\hat G$ and~$\hat H$ be the spaces of equivalence classes of irreducible representations of these groups, and let $\mu_G$ and~$\mu_H$ be the Plancherel measures (see the definition above) on~$\hat G$ and~$\hat H$, respectively. Denote by~$\operatorname{Ind}$ and~$\operatorname{Proj}$ the operations of induction and restriction of irreducible representations, and let $\widehat{\operatorname{Ind}}$ and
~$\widehat{\operatorname{Proj}}$ be the corresponding maps on the simplices~$\operatorname{Simp}(G)$ and~$\operatorname{Simp}(H)$ of probability measures:
 $$\widehat{\operatorname{Ind}}:\operatorname{Simp}(H)\rightarrow \operatorname{Simp}(G), \quad \widehat{\operatorname{Proj}}:\operatorname{Simp}(G)\rightarrow \operatorname{Simp}(H). $$
Then $\widehat{\operatorname{Proj}}\,\mu_G=\mu_H$.
 \end{lemma}
  \begin{proof}
The equality we need to prove is equivalent to the following one:
  $$\operatorname{Proj}(\mu_G)\{\lambda\}=\sum_{\Lambda\supset \lambda}\frac{\{\dim{\Lambda}\}^2}{\operatorname{ord}(G)} =\frac{\{\dim \lambda\}^2}{\operatorname{ord}(H)}=\mu_H(\lambda),$$
which follows immediately from the properties of induced representations. It seems that this relation has not yet appeared in the literature, but for necessary information on induction, see~\cite{Serr}.
    \end{proof}

 \begin{definition}
 A locally finite graded graph is said to be a Plancherel graph if for every vertex~$\lambda$ of level~$n$, the following relation holds:
$$ \frac{\dim(\lambda)}{\sum_{\Lambda:\Lambda \succ \lambda}\dim(\Lambda)}=\frac{d_n}{d_{n+1}},$$
where $d_n=\sum_{\lambda} \dim{(\lambda)}^2$ is the sum over all vertices of level~$n$, i.e., the ratio depends only on the number of the level and is equal to the ratio of the sums of squared dimensions at the corresponding levels.
 \end{definition}

We have proved the following theorem.

\begin{theorem}[theorem-definition]
The Bratteli diagram of a locally finite group $G=\bigcup_n G_n$ is a Plancherel graph. The Plancherel measures on its levels are coherent and determine (by definition) a measure on the space of paths (tableaux), called the Plancherel measure of~$G$.
\end{theorem}

In view of the importance of this notion, we will again discuss it in detail for the symmetric group. The fact that the limiting measure is a measure on sequences of representations (tableaux) poses new questions about this Plancherel measure. For example, a remarkable fact is that it is a Markov measure, like all central measures. We will present the transition and cotransition measures of the corresponding Markov chain. Apparently, there is a necessity to make a list of various properties of graded branching graphs of general locally finite groups, which so far have been studied only for some special groups.

It is of interest to give examples of Plancherel Bratteli graphs different from those arising from groups and to understand to what extent the Plancherel property  is close to being a characterization of Bratteli diagrams of groups. Another question is what Bratteli diagrams are simultaneously Plancherel graphs and differential posets in the sense of Stanley and Fomin~\cite{St,Fo}.

It is easy to check that for commutative locally finite groups,  the new definition of the Plancherel measure coincides with the standard one: it is the Haar measure on the compact group of characters.

So, we associate a Plancherel measure with a locally finite group. One can easily check that this measure does not change if we pass from an inductive sequence of groups~$\{G_n\}$ to a subsequence~$\{G_{n_k}\}$. Thus, the measure is determined by the inductive limit itself. But, of course, an essential modification of the system of embeddings of finite groups in the definition of the inductive limit may change the measure.

Here we have a vast area for various group-theoretic and representation-theoretic questions.

Note that an attempt to define a reasonable probability measure on the set of irreducible representations, e.g., on the set of infinite Young diagrams, is hardly meaningful; the reason is not only in the absence of a separable Borel structure on this set for groups not of type~I, but also in the fact that the notion of a class of measures that ``decompose'' representations of this type itself has little sense.

A convenient tool for the study of asymptotic properties of the Plancherel measure is the ``grand ensemble,'' i.e., the union of all sets of finite diagrams, on which one can define the ``Poissonization'' of the Plancherel measures of finite groups, see~\cite{V96,BOO}.

The above considerations can be applied to a more general case, that of semidirect products of commutative $C^*$-algebras and their automorphism groups. Recall that, in particular, the group  $C^*$-algebra of a locally finite group, being a so-called AF-algebra, is the semidirect product of a commutative (``Gelfand--Tsetlin'') subalgebra  and the group algebra of some (``adic'') group. In this case, the Plancherel measure is defined in a natural way as the invariant measure on the spectrum of this commutative subalgebra, and thus as the spectral measure of the regular representation. Here we can say again that such a definition depends on the realization of the algebra as a semidirect product.

The study of the Plancherel measure for the infinite symmetric group shows to what extent this notion is useful in the study of groups and related combinatorial and other objects.

Below we give an independent, purely combinatorial, definition of the Plancherel measure for the infinite symmetric group. But first we present the main formulas for this measure.

\subsection{The Plancherel measure for the infinite symmetric group}

Recall that the branching of representations of a locally finite group allows one to construct a graph (in the general case, a multigraph) which is a graded graph (Bratteli diagram); a measure on the path space of such a graph is defined in a natural way as the inverse limit of  a coherent system of measures on initial segments of paths. From the probabilistic point of view, we originally have only measures on levels, i.e., know the one-dimensional distributions of the future Markov process, and we must check that there exists a system of transition probabilities that determines a central measure with these one-dimensional distributions (the centrality of a measure is equivalent to its invariance under inner automorphisms).

Consider the infinite Young graph~$\textbf{Y}$, i.e., the graded graph of all finite Young diagrams~$\lambda$, or finite ideals of the lattice~${\Bbb Z}_+^2$. A Young tableau $t=\{t_k\}_{k=1}^n$ with $n$ cells is a monotone sequence of $n$ diagrams, or a path in the Young graph that starts from the one-cell diagram. The space~ ${\cal T}(\textbf{Y})$ of all infinite paths in the Young graph, i.e., the space of infinite Young tableaux, is the inverse limit of the sequence of spaces of finite Young tableaux with respect to the natural projections $P_{n+1}:{\cal T}_{n+1}\rightarrow {\cal T}_n$, $n=1,2\dots$, forgetting the last cell:
$$
{\cal T}(\textbf{Y})=\lim_n ({\cal T}_n, P_n).
$$
This is a Cantor-like compact set equipped with a Borel structure and a tail filtration~$\{\xi_n\}_n$. By definition, two tableaux $t \in {\cal T}_n$ and $t' \in {\cal T}_n$ lie in the same element (class) of the partition~$\xi_n$ if their diagrams coincide for $k>n$, i.e., $t_k=t'_k$ for $k=n+1,n+2, \dots $.

The Plancherel measure on the $n$th level $\textbf{Y}_n$, $n=1,2 \dots$, of the graph~$\textbf{Y}$ is the probability measure~$\mu_n$ defined by the formula
\begin{equation}\label{pl}
    \mu(\lambda)=\frac{\dim(\lambda)^2}{n!}.
\end{equation}
In this and subsequent formulas, $\dim (\lambda)$ is the dimension of the representation~$\pi_{\lambda}$ of the group~$S_n$  canonically associated with the diagram~${\lambda}$; its combinatorial interpretation is as follows: this is the number of paths leading to the diagram~$\lambda$ from the initial vertex~$\emptyset$. This dimension can be found by the Frobenius formula, or the well-known hook-length formula.

By definition, the Plancherel measure~$\mu^n$ on the set~${\cal T}_n$ of all Young tableaux with $n$ cells is a measure whose restrictions to the sets of tableaux with the same diagram~$\lambda$ are uniform, hence it is given by the formula
 $$
 \mu^n(t)=\frac{\dim(\lambda)}{n!},\qquad t\in {\cal T}_n
 $$
   (recall that the number of tableaux with given diagram~$\lambda$ is equal to~$\dim(\lambda)$).

   \begin{lemma}
The collection of measures $\{\mu^n\}_{n=1}^{\infty}$ on the sets of finite tableaux is coherent with respect to the projections~$P_n$; in other words,
    $$\sum_{t' \in {\cal T}_{n+1}: P_{n+1}(t')=t} \mu^{n+1}(t')=\mu^n(t).$$
    \end{lemma}

    \begin{proof}
Upon substituting the values~$\mu^n(t)$ from the previous formula, the desired relation becomes equivalent to the following relation for dimensions of diagrams:
  \begin{equation}\label{ind}
    \sum_{\Lambda:\Lambda\gtrdot \lambda} \dim(\Lambda)=(n+1)\dim(\lambda),
\end{equation}
where $\Lambda$ and $\lambda$ lie at the $(n+1)$th and $n$th level of the Young graph, respectively, and differ by one cell.

The last formula immediately follows from the representation-theoretic meaning of the Young graph as the graph of induced representation. Namely, counting the dimension of such a representation, we have
\begin{equation}\label{ind}
 \sum_{\Lambda\gtrdot \lambda}\pi_{\Lambda}= \operatorname{Ind}_{S_n}^{S_{n+1}}\pi_{\lambda}.
\end{equation}
\end{proof}

   \begin{corollary}
The family of cylinder measures ${\mu^n}$ on the space~{\cal T} of infinite tableaux determines a well-defined Borel probability measure $\mu^{\infty}\equiv {\rm Pl}$, called the Plancherel measure. It is central and nondegenerate.
   \end{corollary}

    \begin{proof}
The existence of this measure follows from the coherence of the cylinder measures with respect to the projections (Kolmogorov's theorem). The centrality is equivalent to the uniformness of the measures on the sets of tableaux with given diagram. The nondegeneracy can be verified directly, since all proper infinite ideals of the lattice~${\Bbb Z}_+^2$ are unions of finitely many rows or columns of the lattice, and the density of every such ideal is equal to zero.
    \end{proof}

This fundamental measure was introduced and named ``Plancherel measure'' in~\cite{VK-77}. The behavior of typical diagrams with respect to the Plancherel measure~$\mu_n$ as $n\to\infty$ is very well studied since the late 1970s; namely, we know the limit shape of diagrams, the asymptotics of the lengths of rows, their fluctuations, and many other things.

For further work with the Plancherel measure, it is convenient to use the language of the theory of Markov chains. Note that it is natural to regard the space~$\cal T$ of Young tableaux as a nonstationary Markov compactum. By construction, ${\rm Pl}$ is a Markov measure (a measure corresponding to a Markov chain) on this space; the state space of this Markov chain at time~$n$ is~${\cal T}_n$, its  \textbf{transition probabilities} are equal to
\begin{equation}
     \operatorname{Prob}\{t_{n+1}=\Lambda|t_n=\lambda\}=
     \frac{\dim (\Lambda)}{(n+1)\dim (\lambda)},
\end{equation}
and its \textbf{cotransition probabilities} (which express the centrality of~${\rm Pl}$) are equal to
\begin{equation}
    \operatorname{Prob}\{t_n=\lambda|t_{n+1}\}=\Lambda\}=\frac{\dim(\lambda)}
     {\dim(\Lambda)}.
\end{equation}

This remarkable Markov chain was already studied in~\cite{VK,Ker}; its new properties (in particular, quasi-stationarity) will be considered later. Here we do not discuss an interesting relation between the Plancherel measure and its generalizations and the notion of a model of representations in the sense of I.~M.~Gelfand~\cite{Gel}. For finite symmetric groups, a model was suggested in~\cite{Klya}. For the infinite symmetric group, this is the representation in the $L^2$ space over the Plancherel measure; the latter plays the role of the ``Koopman representation'' corresponding to the regular representation regarded as a von~Neumann representation. However, we call it the basic representation and not the model, because there can be other candidates for a model. In more detail we consider this question in a paper under preparation.

     \section{Random monotone numberings and the combinatorial meaning of the Plancherel measure}

There is another useful and important interpretation of tableaux and measures on the space of tableaux, in particular, the Plancherel measure, in terms of  random numberings of partially ordered sets (posets). First, we will give general definitions for arbitrary locally finite posets~$P$ with minimal element  $\{\emptyset\} \in P$, and then consider our main example of the lattice~$\Bbb Z_+^2$ with the natural partial order.

A monotone numbering of a poset~$P$ (the main example is $P={\Bbb Z}_+^2$) is a map~$\phi:{\Bbb N}\to P$ from the set of positive integers to~$P$ satisfying the conditions
$$\phi(0)=\emptyset,\qquad \phi(n)\succcurlyeq \phi(m) \Rightarrow  n>m.$$

If $\phi$ is a monotone numbering, then for every~$n$ the set $\{\phi(1),\phi(2),\dots, \phi(n)\}$ is a finite ideal in~$P$ (a Young diagram for the lattice~$\Bbb Z_+^2$).

The graph (Hasse diagram) of all finite ideals of~$P$ ordered by inclusion will be denoted by~$\Gamma(P)$ (for the lattice~$\Bbb Z_+^2$, this is the Young graph). A monotone numbering is an infinite connected path (or tableau) in the graph~$\Gamma(P)$.

The density of an infinite set $I \subset P$ with respect to a numbering ${\phi:{\Bbb N}\rightarrow P}$ is the limit
$$
\lim_n\frac{1}{n}\left|\{i: i<n; \phi(i)\in I\}\right|.
$$

The set of all numberings is a Cantor-like compact set~$\cal N$. A random numbering is, by definition, a Borel probability measure on~$\cal N$. It is determined by a sequence of random elements of~$P$, i.e., a random process $\{\xi_n\}_{n \in \Bbb N}$ with values in~$P$.

A random numbering is said to be {\it nondegenerate} if every proper (possibly, infinite) ideal $I \ne P$ has zero density with probability~1.

A random numbering  with values in~$P$ is said to be a {\it Markov} numbering if the corresponding random process is, in a natural sense, a (nonhomogeneous) Markov chain whose state space at time~$n$ is the set of ideals (diagrams in the case of~${\Bbb Z}^2_+$) with $n$ elements.

A random numbering $\phi$ is said to be {\it central} if for every~$n$ the conditional measure on the set of finite ideals with $n$ elements given that the values $\phi(k)$ with $k>n$ are fixed is uniform with probability~$1$.

A random central numbering is said to be {\it ergodic} if the action of the group of all finite renumberings of the random path (which is well defined by centrality) on the path space equipped with this measure is ergodic.

One can easily see that a central numbering, regarded as a measure on the path space of the Young graph, is exactly a central measure. To describe all (ergodic) central measures on a path space is one of the main problems in the whole theory of graded graphs.

The following theorem is one of the deepest facts of the whole asymptotic representation theory of symmetric groups; this is the second uniqueness theorem in our original list.

\begin{theorem}
There is a unique nondegenerate central random numbering of the lattice~$\Bbb Z_+^2$; the corresponding measure on the space of numberings (paths, or tableaux) is the Plancherel measure. It is Markov and ergodic. Another reformulation: the system of cotransition probabilities of the Plancherel measure determines a unique nondegenerate Markov measure on the space of tableaux.
\end{theorem}

This uniqueness theorem is the most difficult part of the general Thoma theorem~\cite{T} on the characters of the infinite symmetric group, which we discussed in the first section of the paper. By now, there are several proofs of this theorem, but still no purely combinatorial proof. Here are the corresponding references:  \cite{VK,O1,OO1,BG,P}.

The proof closest to a purely combinatorial one is given in~\cite{P}; it uses the following interpretation of the problem suggested by me: it suffices to prove the following combinatorial-probabilistic lemma, which is so far a corollary of the theorem, but is in fact equivalent to the theorem.

\begin{lemma}
For every $k\in {\Bbb N}$ and every $\epsilon>0$, there exists $\delta>0$ and $N\in {\Bbb N}$ such that for every
$n>N$ and for every Young diagram $\lambda \vdash n$ with $\lambda \subset[1,\delta\cdot n]\times[1,\delta\cdot n]$,
$$\rho_k(d_k^{\lambda}(\cdot), {\rm Pl}_k(\cdot))< \epsilon;$$
here $\rho_k$ is an arbitrary metric on the space of probability measures on the set of Young tableaux with $k$ cells, while $d_k^{\lambda}(\cdot)$ and ${\rm Pl}_k(\cdot)$ are the probability distributions on these sets of Young tableaux induced by the Young diagram~$\lambda$ and the Plancherel measure~${\rm Pl}$, respectively.
\end{lemma}

The proof in~\cite{P} uses the theory of symmetric functions and, in a sense, is a continuation of the paper~\cite{OO1} on shifted Schur functions. We will return to the discussion of this proof elsewhere.

Another property of the Plancherel measure in terms of random numberings looks as the quasi-invariance under a lattice shift, more exactly, under passing from a given numbering to the induced one, obtained by omitting the elements of the first row, or the first column, or both. It is easy to prove that such an induction sends the Plancherel measure to an equivalent one. Hence, another equivalent formulation of the uniqueness theorem is as follows.

\begin{proposition}\label{prop1}
The Plancherel measure is the unique central quasi-in\-vari\-ant measure.
\end{proposition}

\subsection{General locally finite posets}

All definitions from the previous section make sense for arbitrary locally finite posets, i.e., arbitrary distributive lattices that are Hasse diagrams of lattices of finite ideals of posets. We mean the notions of monotone numberings, random numberings, and the properties of measures on the space of numberings of a poset (or the space of tableaux, the path space of the graph of ideals): centrality, nondegeneracy, quasi-invariance, Markov property, ergodicity, etc.

The existence of at least one monotone nondegenerate central numbering for an arbitrary poset can usually be easily proved by compactness considerations.

However, the uniqueness of a central nondegenerate measure holds not for all posets.

\begin{definition}
A countable locally finite poset is said to be \textbf{rigid} if a central nondegenerate Markov monotone numbering on this poset exists and is unique.
\end{definition}

The above theorem says that the poset~${\Bbb Z}^2_+$ is rigid.

\begin{problem}
Are the posets ${\Bbb Z}^d_+$ for $d>2$ rigid?
\end{problem}

A direct proof of the uniqueness theorem in any of the suggested formulations, e.g., of Proposition~\ref{prop1}, would reveal deep properties of numberingss of rigid lattices. This would mean a considerable progress in the asymptotic theory of posets and in the theory of traces of  $C^*$-algebras.

Apparently, an example of a nonrigid poset is the poset with the set of elements~${\Bbb Z}^d_+$ and the partial order defined as follows:
$$(x,y)\succ (u,v) \mbox{ if } y>v, \quad (x,0)\succ (u,0) \mbox{ if } x>u,$$
and $(x,y)$ and $(u,v)$ are incomparable if $x\ne u$ and $y$ or $v$ is not zero.

The deformation of the order structure on~${\Bbb Z}_+^2$ that turns rigidity into nonrigidity reminds of percolation phenomena.

All interpretations and properties of the Plancherel measure, including those discussed in the previous sections, are closely intertwined. We will consider another invariance property, but state it in a more general situation.

It is not difficult to show that every central measure on the path space of a graded graph is Markov, but not stationary, because the state space  depends on~$n$. Nevertheless, it has the property of invariance with respect to a ``quasi-shift,'' which we will call the {\it quasi-stationarity} of a nonstationary Markov chain. For the Young graph, this property is a generalization of the combinatorial notion of the Sch\"utzenberger (jeu de taquin) transformation, which in the case of central measures and the Plancherel measure for the lattice~${\Bbb Z}^2_+$ was studied in \cite{VKS,RS,S}. Here we will give a general definition.

We define a (deterministic) map from the path space of a distributive lattice to itself, calling it the transfer. For the Young lattice, this is an infinite analog of the Sch\"utzenberger transformation. It is most convenient to describe the transfer in terms of numberings.

Consider an arbitrary numbering (sequence of elements) $t_1=\emptyset, t_2, t_3, \dots$ of a poset, and associate with it a new numbering $t'_1=\emptyset, t'_2, t'_3, \dots$ as follows: if $t_{n+1}\succ t_n$, then $t'_n=t_n$; and if $t_n$ and $t_{n+1}$ are not comparable, then $t'_n=t_{n+1}$, $n=2,3 \dots$. This is a natural generalization of the ordinary shift; distributivity guarantees that all 2-intervals in the graph are either segments or rhombi. This definition makes sense for more general graphs than Hasse diagrams of distributive lattices, namely, we need only the existence of a ``translation'' of edges for adjacent vertices.

Thus, we can consider the class of Markov chains for which there is a well-defined shift (transfer) in the path space that shifts paths in time in a generalized sense.

\begin{definition}
Consider a Markov (in general, nonstationary) compactum with a well-defined transfer, e.g., the path space of the Hasse diagram of a distributive lattice. A Markov measure on this compactum is said to be quasi-stationary if it is invariant under the transfer.
\end{definition}

\begin{problem}
Every central measure on the Markov compactum of paths in the Hasse diagram of a distributive lattice is quasi-stationary.
\end{problem}

For the  Young graph, this is true and follows from the main result of~\cite{VKS}. Most likely, quasi-stationarity also holds for an arbitrary distributive lattice. The Sch\"utzenberger transformation, i.e., the transfer for the Young graph, has been recently considered in~\cite{RS,S} for other reasons. The main result is that the Sch\"utzenberger transformation on the path space equipped with any central ergodic measure (Thoma measure) is metrically isomorphic to a Bernoulli shift. We will consider these problems in another paper. All these considerations reveal a close relation between this circle of problems and ergodic theory.

\section{The third uniqueness theorem: representations of the infinite symmetric group}

In this section, we essentially consider a corollary of the previous theorems for the representation theory of symmetric groups. In its first form, this corollary is equivalent to the uniqueness theorem and, possibly, has its own independent proof. The second formulation (with a bound) contains an important complement to the uniqueness theorem concerning properties of the Plancherel measure. We state our theorem as a uniqueness theorem for the regular representation of the infinite symmetric group, more exactly, characterize this representation in natural terms.

We will consider traceable factor representations of the group~${\frak S}_{\Bbb N}$, i.e.,  normalized indecomposable characters of~${\frak S}_{\Bbb N}$,  already discussed above. Recall that a character is a function~$\chi(\cdot)$ on~${\frak S}_{\Bbb N}$ satisfying the conditions
$$\chi(e)=1,\quad \chi(gh)=\chi(hg), \quad
||\chi(g_i \cdot g_j^{-1})||_{i,j=1}^n \geq 0 $$
for every $n \in {\Bbb N}$ and any $g_1, \dots g_n \in G$.

Indecomposability means that $\chi(\cdot)$  is an extreme point of the simplex of all characters.

Every indecomposable character~$\chi$ generates a representation~$\pi_{\chi}$ of the group, which is either finite-dimensional, or infinite-dimensional in a von~Neumann factor of type~II$_1$.

A nontautological realization of such representations was given in~\cite{VK-81}, but we will not use it.

The link between the theorem on the exponential function (Section~1) and that on the regular representation and its character (see below) is given by the correspondence
 $$ \exp (z)\dashleftarrow\dashrightarrow \chi_e =\delta_e(g),$$
or, in terms of generating functions,
 $$F_{\chi}(z)=\exp{\sum_n\frac{\chi(g_n)}{n}z^n},$$
 where $g_n\in S_{\infty}$ is a cycle of length $n$; this is a totally positive function.

The most important factor representation of~${\frak S}_{\Bbb N}$ is the (say, left) regular representation. Denote it by~$\Pi$. As mentioned above, it is this representation that corresponds to the Plancherel measure, and the corresponding character is $\chi(g)=\delta_e(g)$, the $\delta$-function at the group identity. It turns out that the regular representation~$\Pi$ and, consequently, its character (i.e., the trace in a von~Neumann factor of type~II$_1$) can be axiomatized as follows.

We will consider restrictions of representations and, in particular, the representation $\Pi$ to finite subgroups ${\frak S}_n$ and define the following sequence:
$$Y_{\chi}(n)\equiv \max \{k<n: \; \exists S_k \subset S_n ,\; \exists h\in H_{\pi_\chi}: \forall g\in S_k \quad \pi_{\chi}(g)h=\pm h \}.$$
In other words, we fix the greatest rank of a subgroup for which the restriction of~$\pi_\chi$ is either
the trivial or the sign representation: $\chi\big|_{S_k}=1$ or  $\chi\big|_{S_k}=\operatorname{sgn}(\cdot)$.

\begin{theorem}[uniqueness and characterization of the regular representation]

{\rm1)} There is a unique factor representation for which   $Y_{\chi}(n)=o(n)$, namely, the regular representation.

{\rm 2)} The regular representation is the unique factor representation for which $Y_{\chi}(n)=O(\sqrt n)$.
\end{theorem}

The second claim follows from the theorem of~\cite{VK-85} which says that the growth of~$Y_{\chi}(n)$ for the regular representation is exactly $Y_{\chi}(n)=2\sqrt n$.

Thus, $Y_n$ grows either linearly, or as $2\sqrt n$, and the latter occurs only in the case of the regular representation~$\Pi$ and the character~$\chi(g)=\delta_e(g)$.

In this form,  the theorem is stronger than the uniqueness theorems, since it implies the total positivity theorem from Section~1 and the uniqueness theorems for the Plancherel measure, but the strengthening is achieved by an additional bound on the growth of the first row of the Young diagram.


\begin{thebibliography}{99}


\bibitem{ASM}M.~Aissen, I.~J.~Schoenberg, A.~M.~Whitney, On the generating functions of totally positive sequences. I, J.
Anal. Math. 2 (1952), 93--103.

\bibitem{Gel}J.~H.~Bernstein, I.~M.~Gelfand, S.~I.~Gelfand, Models of representations of compact Lie groups, Funct. Anal. Appl. 9, No.~4 (1975), 322--324.

\bibitem{BOO}A.~Borodin, A.~Okounkov, G.~Olshanski,  Asymptotics of Plancherel measures for symmetric groups, J. Amer. Math. Soc. 13 (2000), 481--515.

\bibitem{BGl}V.~Buchstaber, A.~Glutsyuk,  Total positivity, Grassmannian and modified Bessel functions,    {\tt arXiv:1708.02154}.

\bibitem{BG}A. Bufetov, V. Gorin, Stochastic monotonicity in Young graph and Thoma theorem, Int. Math. Res. Not.
2015, No.~23, 12920--12940 (2015).

\bibitem{Dix}J.~Dixmier, {\it Les ${\Bbb C}^*$-Algebras et Leurs Representations},
    Paris, 1969.

\bibitem{Edr} A.~Edrei, On the generating functions of totally positive sequences. II, J. Anal. Math. 2 (1952), 104--109.

\bibitem{Fo}S.~Fomin, Duality of graded graphs, J. Algebraic Combin. 3, No.~4 (1994), 357--404.

\bibitem{FZ}S.~Fomin, A.~Zelevinsky, Total positivity: tests and parametrizations, {\tt arXiv:math/9912128}.

\bibitem{GK}F.~R.~Gantmacher, M.~G.~Krein, {\it Oscillation Matrices and Kernels and Small Vibrations of Mechanical Systems},
Revised Edition, AMC Chelsea Publishing, 2002.

\bibitem{Kar}S.~Karlin, {\it Total Positivity}, Vol.~I, Stanford University Press, 1968.

\bibitem{K1}S.~Kerov, Generalized Hall--Littlewood symmetric functions and orthogonal polynomials, in: {\it Representation
Theory and Dynamical Systems},  Adv. Soviet Math. 9, Amer. Math. Soc., Providence, RI, 1992, pp.~67--94.

\bibitem{Ker}S.~Kerov, {\it Asymptotic Representation Theory of the Symmetric Group and its Applications in Analysis},  Transl. Math. Monogr., Vol.~219, Amer. Math. Soc., Providence, RI, 2003.

\bibitem{Klya}A.~A.~Klyachko, Centralizers of involutions and models of the symmetric and general linear groups, in: {\it Studies in Number Theory} [in Russian], Vol.~7, Saratov (1987), pp.~59--64.

\bibitem{Shepp}B.~Logan, L.~Shepp, A variational problem for random Young tableaux, Adv. Math. 26, No.~2 (1977), 206--222.

\bibitem{Lus}G.~Lusztig, A survey in total positivity,
{\tt arXiv:0705.3842}.

\bibitem{Mat}K.~Matveev, Macdonald-positive specializations of the algebra of symmetric functions: Proof of the Kerov
conjecture, to appear in Ann. Math 189, No. 1  (2019).

\bibitem{O1}A.~Okounkov, Thoma's theorem and representations of the infinite bisymmetric group, Funct. Anal. Appl. 28, No.~2 (1994), 100--107.

\bibitem{OO1} A.~Okounkov, G.~Olshanski, Shifted Schur functions, St. Petersburg Math. J. 9, No.~2 (1998), 239--300.	

\bibitem{P}F.~V.~Petrov, The asymptotics of traces of paths in the Young and Schur graphs, Zap. Nauchn. Semin. POMI
468 (2018), 126--137.

\bibitem{RS} D.~Romik, P.~Sniady,
Jeu de taquin dynamics on infinite Young tableaux and second class particles,  Ann. Probab. 43, No.~2 (2015), 682--737.

\bibitem{Sch}I.~J.~Schoenberg, Some analytical aspects of the problem of smoothing, in: {\it Courant Anniversary Volume ``Studies and Essays,''} New York, 1948, pp.~351--370.

\bibitem{Serr}J.-P.~Serre, {\it Representations Lineaires Des Groupes Finis}, Paris, 1967.

\bibitem{S}P.~Sniady. Robinson--Schensted--Knuth algoritm, jeu de taquin, and Kerov--Vershik measure on infinite tableaux,
    SIAM J. Discrete Math.  28, No.~2 (2014), 598--630.

\bibitem{St}R.~Stanley, Differential posets,
J. Amer. Math. Soc. 1, No.~4 (1988), 919--961.

\bibitem{T}E.~Thoma, Die unzerlegbaren, positiv-definiten Klassenfunktionen der abzahlbar unendlichen, symmetrischen
Gruppe, Math. Z. 85 (1964), 40--61.

\bibitem{VStat}A.~M.~Vershik, A partition function connected with Young diagrams, J. Sov. Math. 47, No.~2  (1989), 2379--2386.

\bibitem{V-94}A.~Vershik, Asymptotic combinatorics and algebraic analysis, in: {\it Proc. International Congress of Mathematicians} (Zurich, 1994), Vol.~II,  Birkhauser, Basel, 1995, pp.~1384--1394.

\bibitem{V96} A.~M.~Vershik, Statistical mechanics of combinatorial partitions, and their limit configurations, Funct. Anal. Appl. 30, No.~2 (1996), 90--105.

\bibitem{VPr}A.~M.~Vershik, The problem of describing central measures on the path spaces of graded graphs, Funct. Anal. Appl. 48, No. 4 (2014), 1--20.

\bibitem{VK-77}A.~M.~Vershik, S.~V.~Kerov, Asymptotics of the Plancherel measure of the symmetric group and the limiting form of Young tableaux, Sov. Math. Dokl. 18 (1977), 527--531.

\bibitem{VK-81}A.~M.~Vershik, S.~V.~Kerov, Characters and factor representations of the infinite unitary group, Sov. Math. Dokl. 26 (1982), 570--574.

%\bibitem{VK-82D}A.~M.~Vershik, S.~V.~Kerov, Characters and factor representations of the infinite unitary group, Sov. Math. Dokl. 26 (1982), 570--574.

\bibitem{VK-85}A.~M.~Vershik, S.~V.~Kerov, Asymptotics of maximal and typical dimensions of irreducible representations of a symmetric group, Funct. Anal. Appl. 19 (1985), 21--31.

\bibitem{VKS}A.~M.~Vershik, S.~V.~Kerov, The characters of the infinite symmetric group and probability properties of
the Robinson--Schensted--Knuth correspondence, SIAM J. Alg. Disc. Meth. 7 (1986), 116--123.

\bibitem{VK}A.~M.~Vershik, S.~V.~Kerov, Locally semisimple algebras. Combinatorial theory and the K-functor,  J. Sov. Math. 38  (1987), 1701--1733.

\bibitem{VSh1}A.~M.~Vershik, A.~A.~Shmidt, Limit measures arising in the asymptotic theory of symmetric groups. I, Theory Probab. Appl. 22, No.~1  (1977), 70--85.

\bibitem{VSh2}A.~M.~Vershik, A.~A.~Shmidt, Limit measures arising in the asymptotic theory of symmetric groups. II, Theory Probab. Appl. 23  (1978), 36--49.

\bibitem{W}A.~M.~Whitney, A reduction theorem for totally positive matrices, J. Anal. Math. 2 (1952), 88--92.





\end{thebibliography}
\end{document}